\documentclass[12pt,letterpaper,reqno,oneside]{amsart}
\usepackage{amssymb}

\usepackage{geometry}

\usepackage{amssymb}
\usepackage{amsmath}
\usepackage{tikz}
\usepackage{hyperref}
\usepackage{microtype}
\hypersetup{
	pdfstartview={XYZ null null 1.00}, 
	pdfpagemode=UseNone, 
	colorlinks,
	breaklinks, 
	linkcolor=blue,
	urlcolor=blue, 
	anchorcolor=blue,
	citecolor=blue
}
\usepackage[capitalise]{cleveref}
\usepackage[shortlabels]{enumitem}

\newtheorem{theorem}{Theorem}[section]
\newtheorem{question}[theorem]{Question}
\newtheorem{lemma}[theorem]{Lemma}

\theoremstyle{definition}

\title{Exponential bounds for monochromatic sums equal to products}
\author{Matt Bowen}
\address{Mathematical Institute, Radcliffe Observatory Quarter, Woodstock Road, Oxford OX2 6GG, England}
\email{matthew.bowen@maths.ox.ac.uk}
\thanks{The author is supported by Ben Green's Simons Investigator Grant number 376201}
\date{August 2024}

\begin{document}

\begin{abstract}
    We show that any $r$-coloring of $\{1,...,r^{r^{r^{3r}}}\}$ contains monochromatic sets $\{a,b,a+b,x,y,xy\}$ with $a+b=xy.$
\end{abstract}

\maketitle

\section{Introduction}

Answering a question of Csikv\'{a}ri, Gyarmati and S\'{a}rk\H{o}zy \cite{csikvari2012density}, Bergelson and Hindman independently gave proofs of the following result.

\begin{theorem}\cite{hindman2009monochromatic}\cite{bergelson2010ultrafilters}\label{thm1}
    Any finite coloring of $\mathbb{N}$ contains monochromatic sets $\{a,b,x,y\}$ with $a+b=xy.$
\end{theorem}

In subsequent years many generalizations and alternate proofs of this fact have been discovered; see for example \cite{baglini2014partition,baglini2015nonstandard,bowen2022monochromatic,di2018ramsey,farhangi2021partition}.  In contrast to the simplicity of its statement, to our knowledge all previous proofs of Theorem \ref{thm1} rely on infinatary arguments.  

This has limited the accessibility of the result. For example, in their textbook \cite[Theorem 9.47]{landman2014ramsey}, Landman and Robertson advertised Theorem \ref{thm1} but did not include its proof, stating ``The tools that are used to prove these results are (far) beyond the scope of this book."

The infinatary nature of the aforementioned arguments also do not provide any quantitative bounds on the Ramsey numbers for the configuration from Theorem \ref{thm1}.  This is in contrast to the variant of the result in the modular setting, where the quantitative aspects of the problem are well understood: S\'{a}rk\H{o}zy \cite{sarkozy2005sums} gave good upper bounds for the analogue of Theorem \ref{thm1} in $\mathbb{Z}/p\mathbb{Z},$ and his argument was later extended to provide similar bounds in $\mathbb{Z}/{p^r}\mathbb{Z}$ by Gyarmati and S\'{a}rk\H{o}zy \cite{gyarmati2008equations} and extended further still to $\mathbb{Z}/m\mathbb{Z}$ by Pach \cite{pach2012ramsey}. 

In this note we give a proof of (extensions of) Theorem \ref{thm1} that only relies on Ramsey's theorem and gives an exponential upper bound.

\begin{theorem}\label{thm: m}
    Let $H(r)$ be the smallest integer such that any $r$-coloring of $\{1,...,H(r)\}$ contains a monochromatic set $\{a,b,a+b,x,y,xy\}$ with $a+b=xy.$  Then 

    $$2^{2^r}<H(r)<r^{r^{r^{3r}}}.$$

    More generally, if $R(k;r)$ is the $r$-color Ramsey number for $k-cliques,$ then any $r$-coloring of $R(n+1;r^{R(m+1;r)})^{R(m+1;r)}$ contains a monochromatic set $\{a_i, \sum_{i=1}^na_i, x_j, \prod_{j=1}^mx_j: i\leq m, j\leq n\}$ with $\sum_{i=1}^na_i=\prod_{j=1}^mx_j.$

\end{theorem}

Neither the listed bounds nor the configurations considered are the best possible that follow from our argument, but we will present Theorem \ref{thm: m} as stated for ease of exposition.  The lower bound for $H(r)$ is just a (sub-optimal) lower bound for the multiplicative Schur number controlling the configuration $\{x,y,xy\}$ and makes no reference to the additive content of the theorem.  Notice that in Csikv\'{a}ri, Gyarmati and S\'{a}rk\H{o}zy's original question this set was not required to be monochromatic, so it is plausible that there is an upper bound for Theorem \ref{thm1} that is better than the lower bound for Theorem \ref{thm: m}.

Finally, we end the introduction by pointing out one of the simplest results of this type that we still do not know how to provide reasonable bounds for.

\begin{question}
    Let $B(r)$ be the smallest integer such that any $r$-coloring of $\{1,...,B(r)\}$ contains sets $X=\{x_1,x_2,x_1+x_2\}$ and $Y=\{y_1,y_2,y_1+y_2\}$ with $X\cup Y\cup XY$ monochromatic. Is there an exponential upper bound (i.e., a tower of height independent of $r$) for $B(r)?$
\end{question}

The fact that $B(r)$ is finite for all $r$ follows from a result of Beigleb\"ock, Bergelson, Hindman, and Strauss \cite[Theorem 3.7]{beiglbock2008some} which was proven using ultrafilter techniques.  Our proof of Theorem \ref{thm: m} can be adapted to give an elementary proof of this result, but at the expense of requiring a refinement step that leads to bounds with a tower type dependence on $r.$
    
\section{Proof of Theorem \ref{thm: m}}

Before beginning, we fix a small amount of notation.  We denote by $R(k;r)$ the smallest integer such that for any $r$-coloring of pairs $1\leq i<j\leq R(k;r)$  there are $i_1<i_2<...<i_k$ such that all pairs $(i_g,i_h)$ with $g<h$ are monochromatic. Given a sequence $s_1,s_2,...\in \mathbb{N},$ for $i\leq j$ set $s_{i,j}=s_i\cdot s_{i+1}\cdot...\cdot s_{j-1},$ where by convention $s_{i,i}=1.$

Our argument is an extension of the standard proof of Schur's theorem based on Ramsey's theorem as found in \cite{graham1991ramsey}.  To motivate the main trick, observe that if $S=\{a,b,a+b\},$ then for any $n\in\mathbb{N}$ we have $nS=\{a',b',a'+b'\}.$  This `dilation invariance' will allow us to reduce the $a+b=xy$ Ramsey problem directly to the standard proof of the $xy=z$ Ramsey problem with only a relatively small loss in the bounds we obtain.  This reduction is the content of Lemma \ref{claim}.  A similar trick could be used for any other finite dilation invariant Ramsey family (such as any of the configurations obtained from Rado's theorem), but we will focus on Theorem \ref{thm: m} for ease of exposition.

\begin{lemma}\label{claim}
    In any finite coloring of $\{1,...,R(n+1;r^{M})^M\},$ there are finite sets $S_1,...,S_M$ of size $n$ with distinguished elements $s_i\in S_i$ such that for any $i\leq j\leq M$ the color of $s_{i,j}\cdot s$ for $s\in S_{j}\cup \{\sum_{s'\in S_{j}}s'\}$ depends only on $i$ and $j$ and not on the choice of $s\in S_{j}\cup \{\sum_{s'\in S_{j}}s'\}.$
\end{lemma}

Before proving Claim \ref{claim}, let us use it to prove the upper bounds from Theorem \ref{thm: m}.

\begin{proof}[Proof of the upper bounds from Theorem \ref{thm: m}]
    This will follow from applying Ramsey's Theorem to the sequence of $s_i$ from Lemma \ref{claim} as in the standard proof of Schur's Theorem.  To be more precise, set $M=R(m+1,r)$ and fix an $r$-coloring of $\{1,...,R(n+1; r^{M})^M\}$.  Let $s_i\in S_i$ be as in Lemma \ref{claim}.  For $1\leq i<j\leq M,$ color pair $(i,j)$ based on the color of $s_{i,j}.$  By our choice of $M$, there are $i_1<...<i_{m+1}$ such that all pairs $(i_g,i_h)$ for $g<h$ are monochromatic.  Let $l=i_{m+1}-1$ and let $a_1,...,a_n$ be an arbitrary enumeration the elements of $s_{i_1,l}\cdot S_{l}.$  For $j<m$ let $x_j=s_{i_j,i_{j+1}}$ and set $x_m=s_{i_m,l}\cdot \sum_{s'\in S_{l}}s'.$  Finally, observe that $\sum_{i=1}^na_i=\prod_{j=1}^mx_j$ and that these terms together with the $a_i$ and $x_j$ are monochromatic by our choice of $i_1,...,i_{m+1}$ and Lemma \ref{claim}. 

    The bound in the $m=n=2$ case now follows from the (weak) upper bound $R(3;r)\leq r^r$ \cite{chung1983survey}.
\end{proof}

We now prove Lemma \ref{claim}, which once again follows easily from Ramsey's theorem.

\begin{proof}
    Fix an $r$-coloring of $\{1,...,R(n+1;r^{M})^M\}.$  For $i\in [M]$ let $N_i=R(n+1; r^{i}).$  We inductively define $s_i\in S_i\subset [N_i]$ as follows.  In the base step, apply Ramsey's theorem to $[N_1]$ as in the usual proof of Schur's theorem to find a set $S_1\subset N_1$ with $S_1\cup \{\sum_{s'\in S_1}s'\}$ monochromatic.  Choose $s_1\in S_1$ arbitrarily.  Having defined $s_1,...,s_i,$ to find $S_{i+1}$ consider the $r^{i+1}$-coloring of $[N_{i+1}]$ where $n\in N_{i+1}$ is colored based on the tuple listing the color of $n$ and $s_{j,i+1}\cdot n$ for each $j\leq i.$  By our choice of $N_{i+1}$ and Ramsey's theorem, there is an $S_{i+1}\subset N_{i+1}$ such that $S_{i+1}\cup \{\sum_{s'\in S_{i+1}}s'\}$ is monochromatic with respect to this coloring, and in particular satisfies the conditions from Lemma \ref{claim}.  Choose $s_{i+1}\in S_{i+1}$ arbitrarily and we are done.  
\end{proof}

\vspace{3mm}

We end by giving the lower bound from Theorem \ref{thm: m}.  Let $S_+(r)$ be the smallest integer such that any $r$-coloring of $\{1,...,S_+(r)\}$ contains a monochromatic set $\{a,b,a+b\}$ and $S_\times(r)$ the smallest integer such that any $r$-coloring of $\{2,...,S_\times(r)\}$ contains a monochromatic set $\{x,y,xy\}$.  Any lower bound for $S_\times(r)$ gives a lower bound for Theorem \ref{thm: m}, so we establish one below.  The argument we give here is surely known (see e.g. \cite{142793}), but we could not find a reference to it in the literature.   

\begin{lemma}
    $S_\times(r)\geq 2^{S_+(r)-1}.$  In particular, $S_\times(r)\geq 2^{2^r}.$
\end{lemma}

\begin{proof}
    Consider an $r$-coloring $c_+$ of $\{1,...,S_+(r)-1\}$ with no monochromatic sets $\{a,b,a+b\}.$  Let $P$ be the set of primes.  Define a coloring $c_\times$ of $\{2,...,2^{S_+(r)-1}\}$ where the $c_\times$ color of $n=\prod_{p\in P}p^{m_{p,n}}\leq 2^{S_+(r)-1}$ is the $c_+$ color of $\sum_{p\in P}m_{p,n}.$  If there is a set $\{x,y,xy\}$ which is monochromatic with respect to $c_\times,$ then for $a=\sum_{p\in P}m_{p,x}$ and $b=\sum_{p\in P}m_{p,y}$ we have have that $\{a,b,a+b\}$ is monochromatic with respect to $c_+,$ a contradiction.

    The in particular part of the Lemma now follows from the well known bound $S_+(r)> \frac{3^r-1}{2}$ \cite{chung1983survey}.
\end{proof}

\bibliographystyle{amsalpha}
\bibliography{bib}

\providecommand{\bysame}{\leavevmode\hbox to3em{\hrulefill}\thinspace}
\providecommand{\MR}{\relax\ifhmode\unskip\space\fi MR }
\providecommand{\MRhref}[2]{%
  \href{http://www.ams.org/mathscinet-getitem?mr=#1}{#2}
}
\providecommand{\href}[2]{#2}
\begin{thebibliography}{BBHS08}

\bibitem[Bag]{baglini2014partition}
Lorenzo~Luperi Baglini, \emph{Partition regularity of nonlinear polynomials: a nonstandard approach}, The Seventh European Conference on Combinatorics, Graph Theory and Applications: EuroComb 2013, Springer, pp.~407--412.

\bibitem[Bag15]{baglini2015nonstandard}
\bysame, \emph{A nonstandard technique in combinatorial number theory}, European Journal of Combinatorics \textbf{48} (2015), 71--80.

\bibitem[BBHS08]{beiglbock2008some}
Mathias Beiglb{\"o}ck, Vitaly Bergelson, Neil Hindman, and Dona Strauss, \emph{Some new results in multiplicative and additive ramsey theory}, Transactions of the American Mathematical Society \textbf{360} (2008), no.~2, 819--847.

\bibitem[Ber10]{bergelson2010ultrafilters}
Vitaly Bergelson, \emph{Ultrafilters, ip sets, dynamics, and combinatorial number theory}, Ultrafilters across mathematics \textbf{530} (2010), 23--47.

\bibitem[Bow22]{bowen2022monochromatic}
Matt Bowen, \emph{Monochromatic products and sums in $2 $-colorings of $\mathbb {N} $}, arXiv preprint arXiv:2205.12921 (2022).

\bibitem[CG83]{chung1983survey}
Fan~RK Chung and Charles~M Grinstead, \emph{A survey of bounds for classical ramsey numbers}, Journal of Graph Theory \textbf{7} (1983), no.~1, 25--37.

\bibitem[CGS12]{csikvari2012density}
P{\'e}ter Csikv{\'a}ri, Katalin Gyarmati, and Andr{\'a}s S{\'a}rk{\"o}zy, \emph{Density and ramsey type results on algebraic equations with restricted solution sets}, Combinatorica \textbf{32} (2012), no.~4, 425--449.

\bibitem[DNB18]{di2018ramsey}
Mauro Di~Nasso and Lorenzo~Luperi Baglini, \emph{Ramsey properties of nonlinear diophantine equations}, Advances in Mathematics \textbf{324} (2018), 84--117.

\bibitem[FM21]{farhangi2021partition}
Sohail Farhangi and Richard Magner, \emph{On the partition regularity of $ ax+ by= cw{m}z^{n}$}, arXiv preprint arXiv:2105.02190 (2021).

\bibitem[GRS91]{graham1991ramsey}
Ronald~L Graham, Bruce~L Rothschild, and Joel~H Spencer, \emph{Ramsey theory}, vol.~20, John Wiley \& Sons, 1991.

\bibitem[GS08]{gyarmati2008equations}
Katalin Gyarmati and Andr{\'a}s S{\'a}rk{\"o}zy, \emph{Equations in finite fields with restricted solution sets. ii (algebraic equations)}, Acta Mathematica Hungarica \textbf{119} (2008), no.~3, 259--280.

\bibitem[Hin09]{hindman2009monochromatic}
Neil Hindman, \emph{Monochromatic sums equal to products in n}, INTEGERS \textbf{11} (2009), no.~2011.

\bibitem[hpg]{142793}
Adam P.~Goucher (https://mathoverflow.net/users/39521/adam-p goucher), \emph{Geometric van der waerden theorem}, MathOverflow, URL:https://mathoverflow.net/q/142793 (version: 2013-09-21).

\bibitem[LR14]{landman2014ramsey}
Bruce~M Landman and Aaron Robertson, \emph{Ramsey theory on the integers}, vol.~73, American Mathematical Soc., 2014.

\bibitem[Pac12]{pach2012ramsey}
P{\'e}ter~P{\'a}l Pach, \emph{Ramsey type results on the solvability of certain equation in zm}, INTEGERS \textbf{12} (2012), 2.

\bibitem[S{\'a}r05]{sarkozy2005sums}
Andr{\'a}s S{\'a}rk{\"o}zy, \emph{On sums and products of residues modulo p}, Acta Arithmetica \textbf{4} (2005), no.~118, 403--409.

\end{thebibliography}

\end{document}